\documentclass[11pt]{amsart}
\usepackage{amscd,amssymb,amsmath,enumerate}
\usepackage{amsmath,amsthm}
\usepackage{hyperref}

\usepackage{graphics}
\usepackage{latexsym}
\usepackage[all]{xy}
\usepackage{psfrag}

\xyoption{matrix} \xyoption{arrow}
\usepackage{parskip}
\usepackage[normalem]{ulem}
\usepackage{float}
\newtheorem{proposition}{Proposition}[section]
\newtheorem{theorem}[proposition]{Theorem}
\newtheorem{lemma}[proposition]{Lemma}
\newtheorem{corollary}[proposition]{Corollary}

\newtheorem{definition}[proposition]{{Definition}}

\newtheorem{remark}[proposition]{{Remark}}

\newtheorem{Example}[proposition]{Example}

\newcommand{\cB}{{\mathcal B}}
\newcommand{\cD}{{\mathcal D}}

\newcommand{\cI}{{\mathcal I}}

\newcommand{\K}{{ K}}

\newcommand{\kq}{KQ/C}

\newcommand{\End}{\operatorname{End}\nolimits}

\newcommand{\smod}{\operatorname{\mathrm{mod}}\nolimits}

\renewcommand{\dim}{\operatorname{dim}\nolimits}

\newcommand{\lcl}{\Lambda/C^f}

\newcommand{\extto}{\xrightarrow}

\usepackage{color}

\begin{document}

\title[The Commuting Algebra]
{The Commuting Algebra}

\author[Green]{Edward L.\ Green}
\address{Edward L.\ Green, \\ Department of
Mathematics\\ Virginia Tech\\ Blacksburg, VA 24061\\
USA}
{\email{green@math.vt.edu}
\author{Sibylle Schroll}
\address{Insitut f\"ur Mathematik, Universit\"at zu K\"oln, Weyertal 86-90, K\"oln, Germany and
Institutt for matematiske fag, NTNU, N-7491 Trondheim, Norway}
\email{schroll@math.uni-koeln.de}

\thanks{This work was partially  supported through the DFG through the project SFB/TRR 191 Symplectic Structures in Geometry, Algebra and Dynamics (Projektnummer 281071066-TRR 191).}

\maketitle

\begin{abstract}
Let $KQ$ be a path algebra, where $Q$ is a finite quiver and $K$ is a field. We study $KQ/C$ where
$C$ is the two-sided ideal in $KQ$ generated by all differences of parallel paths in $Q$.  We show
that $KQ/C$ is always finite dimensional and its global dimension is finite. Furthermore, we prove that 
$KQ/C$ is Morita equivalent to an incidence algebra.
\\
The paper starts with a more general setting, where $KQ$ is replaced by $KQ/I$ with $I$ a two-sided
ideal in $KQ$.
\end{abstract}

\section{Introduction}

In the study of the  representation theory of  finite dimensional $K$-algebras  with $K$ a field, the 
algebras one encounters  often are of the form $KQ/I$, where $I$ is an admissible ideal; that is $J^n\subseteq I\subseteq J^2$, for some positive integer $n$,  and $J$ is the ideal in $KQ$ generated by the arrows of $Q$.  It is not unreasonable to say that $J$ plays a `special' role in
the theory.

Our overall goal is to show that there is another `special'  ideal in  a path algebra that connects any (not necessarily finite dimensional) path algebra of a finite quiver $Q$ with a subring of a matrix ring and with an incidence algebra.  It is well known that incidence algebras, partially ordered sets and Hasse diagrams are all interrelated and that their representation theory is well-studied and well-understood.  For some examples of classical as well as recent work in that direction, see \cite{B, FI, IZ, IM, K, NR}. 
If $\Lambda =KQ/I$ is a not necessarily finite dimensional algebra, we consider a different type of 'special'  ideal, $C$, in $\Lambda$  which has the property that  $(KQ/I)/C$ is always finite dimensional and contains information about $\Lambda$.  There are no restrictions 
on $Q$ nor $I$ other than $Q$ is a finite quiver.
 In particular,  $I$ need not be finitely generated and $KQ/I$ need not be left nor right Noetherian.  
 For the special case $I=0$ and $\Lambda =KQ$,  we prove that $KQ/C$ has finite global dimension.
Moreover, in this case, we show that a basic finite dimensional algebra, Morita equivalent to $KQ/C$, is an incidence algebra.  
This paper provides a detailed analysis of $\Lambda/C$,with special attention
given to the case $\Lambda=KQ$.

We summarize the results of the paper.  In Section 2, we define a quasi-commuting ideal and
quasi-commuting algebras.  Theorem \ref{thm:nh-dim_1} proves that all  
quasi-commuting algebras are finite dimensional   $K$-algebras. Section 3
further studies properties of quasi-commuting algebras. In Section 4, we
turn our attention to  the relation on the vertex set of $Q$
given by $v\sim w$ if there are paths in $Q$ from $v$ to $w$ and from 
$w$ to $v$. This leads to showing
 that if $I=0$ and all the coefficients in the  commuting relations defining the ideal $C$ are equal to 1,
  the  commuting algebra is in block matrix form
(see Definition \ref{def:block} and Theorem \ref{thm:struct_thm})  In particular, we see that a   commuting algebra is a subring of a full matrix ring.
   In Section 5, we apply Morita equivalence theory to find a basic algebra Morita equivalent to
the commuting algebra.  We call such an  algebra a skeleton of $Q$.  We show
that the vertex set  of a skeleton can be partially ordered (Theorem \ref{prop:partial}) and hence  skeletons 
have
finite global dimension (Theorem \ref{thm:fin_dim}).  It follows that commuting algebras
have finite global dimension.  Thus starting with a finite quiver $Q$,  a skeleton of $KQ$ is
always an incidence algebra.

\section{Quasi-commuting algebras}
We begin by recalling the definition of\emph{ parallel paths} in a quiver.
Let $p$ and $q$ be (finite) paths in $Q$.  Then we say $p$ is \emph{parallel} to $q$, denoted $p\|q$, if there exist vertices $v,w$ in $Q$
such that $vp=p, vq=q$, $pw=p$, and $qw=q$. It is convenient to let $\cB$ denote the set of finite paths in $Q$. Note that $\cB$ includes  the paths of length zero; namely,
the vertices of $Q$, and that $\cB$ is an infinite set if and only there is an oriented cycle
of length $\ge 1$ in $Q$.
 
We fix an ideal $I$ in $KQ$ that is contained in the ideal in $KQ$ generated
by paths of length 2. We let $KQ/I$ be denoted by $\Lambda$,  the canonical surjection $KQ\to \Lambda$ by $\pi$ and let $f: \cB \to K^*$ be a set map.

\begin{definition}{\rm     Let \emph{the quasi-commuting ideal of $\Lambda$ and $f$},
denoted by $C^f$,
 be the ideal in 
$\Lambda$ generated by 
all $f(p)\pi(p)-f(q)\pi(q)$,  where   $p,q\in\cB$ and $p\| q$. We call $\Lambda/C^f$  \emph{the quasi-commuting algebra of $\Lambda$ and $f$}. In the special case where $f$ is the constant map equal to $1$,  we call 
$\Lambda/C^f$, \emph{the commuting algebra of $\Lambda$} and   $C=C^f$ the commuting ideal in $\Lambda$. }
\end{definition}
Denote the canonical surjection $\Lambda\to \Lambda/C^f$ by $\rho$.

 \section{Properties of 	quasi-commuting algebras}      
The following result is fundamental to the study of the structure of quasi-commuting algebras.

\begin{theorem}\label{thm:nh-dim_1} Let $Q$ be a quiver with $n$ vertices. Then keeping the notation above, we have 
\begin{enumerate}
\item if $v_i$ and $v_j$ are vertices in $Q$, then $\dim_K(v_i\Lambda/C^f v_j)\le 1$, for
all  $1\le i,j\le n$.
\item  Every  quasi-commuting algebra of $\Lambda$ is finite dimensional, with
dimension over $K$ no greater than $n^2$, where $n$ is the number of vertices of $Q$.

\end{enumerate}

\end {theorem} 
\begin{proof}
We begin with a proof of (1).  Let $B$ be the $K$-basis of paths in  $v_iKQv_j$, for vertices
$v_i$ and $v_j$ of $Q$. In particular, $B=v_i\cB v_j$.
Let $\rho\pi(B)$ denote the set $\{\rho\pi(p) | p\in B\}$.  Clearly
$\rho\pi(B) $ generates $v_i(\Lambda/C^f)v_j$.  Suppose that
 $v_i(\Lambda/C^f_{\Lambda})v_j\ne 0$. It follows that there is some $p^*\in B$
such that $\rho\pi(p^*)\ne 0$.  We show that $\rho\pi(p^*)$ generates $v_i(\Lambda/C^f)v_j$ as a $K$-vector space;  thus showing $\dim_K(v_i(\Lambda/C^f)v_j)=1$.
 
We see that $\rho\pi(B)$ generates $v_i\lcl v_j$. Let $p$ be a path from $v_i$ to $v_j$
in $Q$.   Then $\pi(f(p^*)p^*-f(p)p)\in C^f$. Hence $[f(p^*)/f(p)]\rho\pi(p^*)=\rho\pi(p)$ and we conclude that $\rho\pi(p^*)$ generates $v_i\lcl v_j$.

To prove (2), note that $\Lambda/C^f =\oplus_{i=1}^n\oplus_{j=1}^n
v_i(\Lambda/C^f)v_j$.   The result follows from (1).

\end{proof}

The next result will be used frequently.

\begin{corollary}\label{prop:path_implies_arrows}
Let $Q$ be a quiver and $I$ an ideal in $KQ$ contained
in the ideal generated by paths of length $2$.  Suppose $v,w$ are vertices in $Q$ (not necessarily
distinct).  Let $f\colon \cB\to K^*$. The following statements are equivalent
\begin{enumerate}
\item There is a path $p$  from $v$ to $w$ in $Q$ such that $\rho\pi(p)\neq 0$
\item $\dim_Kv(\Lambda/C^f)w\ne 0$
\item$\dim_Kv(\Lambda/C^f)w=1$.
\end{enumerate}

\end{corollary}

\begin{proof}  Parts (2) and (3) are seen to be equivalent by
Theorem \ref{thm:nh-dim_1}.

 It is obvious that (1) implies part (2).

Now assume part (2) holds.  Then, since $v_i\Lambda /C^f v_j \ne 0$, there is a path 
$p$ from $v_i$ to $v_j$ such that $\rho\pi(p)\neq 0$ and we are done.
\end{proof}

\begin{remark}{\rm Note that if $I = 0$, this implies that $\rho $ is the identity map and thus  there is a non-zero path from $v$ to $w$ in $Q$ if and only if $v(\lcl)w = v(KQ/C^f)w \neq 0$. 
}\end{remark}

Combining Theorem~\ref{thm:nh-dim_1} and Corollary~\ref{prop:path_implies_arrows}, we have the following. 

\begin{corollary}\label{prop:out} Let $Q$ be a quiver and $I$ an ideal in $KQ$ contained
in the ideal generated by paths of length $2$.  Suppose $v,w$ are vertices in $Q$ (not necessarily
distinct).  The following statements are equivalent
\begin{enumerate}
\item $\rho\pi(p)=0$, for all paths $p$  from $v$ to $w$ in $Q$
\item For all paths $q$ from $v$ to $w$ in $Q$, $q\in (vC^f w)+I$
\item $\dim_K(v(\lcl)w)=0$.
\end{enumerate}
\end{corollary}

%
%

\begin{corollary}\label{prop:basis}
Let $v$ and $w$ be two, not necessarily distinct, vertices in $Q$.
 If $p$ is a path from $v$ to $w$ such that $\rho\pi(p)\ne 0$, then $\rho\pi(p) $ is a $K$-basis
for $v(KQ/C^f)w$. If $q$ is another path from $v$ to $w$ such that 
$\rho\pi(q) \ne 0$, then $f(q)\rho\pi(q) =f(p)\rho\pi(p)$.
\end{corollary}
\begin{proof} Using 
Theorem \ref{thm:nh-dim_1}
 the result  follows.
\end{proof}

The following result is easy to prove and we include a proof for completeness.
\begin{lemma}\label{lem:lin_alg}  Let $V$ be a $K$-vector space with $K$-basis $\cB^*=\{b_i\mid i\in\cI\}$.    Let $b\in \cB^*$ and let $f: \cB^* \to K^*$. Define $C$ to be the $K$-subspace of $V$ generated by  the set $\{f(b_i) b_i - f(b) b\mid  i
\in\cI\}$.
Then  $b\not\in C$.\end{lemma}

\begin{proof} If $b\in C$ then $b$ is a finite linear combination of elements of the form
$f(b_i) b_i- f(b) b$; say $b=\sum_{b_i \neq b} \alpha_i(f(b_i) b_i- f(b) b)$. \\  Thus we obtain
\[b+(\sum_{b_i \neq b} \alpha_i)f(b)b=\sum_{b_i \neq b} ( \alpha_if(b_i)b_i).\] This contradicts $\cB^*$ is a linearly independent set.\end{proof}

As noted in the proof of Theorem \ref{thm:nh-dim_1}, the quasi-commuting algebra, $\lcl$, has
a  direct sum  decomposition

\[(*)\quad \quad \Lambda/C^f = \oplus_{i=1}^n\oplus_{j=1}^n
 v _i(\Lambda/C^f) v_j.\]
 
We use this observation in the next section, and
we end this section, with the following result:

\begin{lemma}\label{lem:surjection}Let $L$ and $I$ be ideals in $KQ$ such that $L\subseteq I\subseteq J^2$, the ideal  of $KQ$ generated by paths of length 2. 
Then for any $f: \cB \to K^*$,  the quasi-commuting algebra of $KQ/L$  and $f$  maps onto the quasi-commuting algebra of $KQ/I$ and  $f$. 
\end{lemma}
\begin{proof}
Let 
$\pi_I \colon KQ\to KQ/I$,  $\pi_L \colon KQ\to KQ/L$, $\rho_I\colon KQ/I\to (KQ/I)/C_I^f$, and
$\rho_L\colon KQ/L\to (KQ/L)/C_L^f$
be the canonical surjections.  Then $C_L^f$ is generated
by $f(p) \pi_L(p)-f(q)\pi_L(q)$ for parallel paths $p$ and $q$ and $\pi_L(f(p) p-f(q) q)\in C_L^f$
and $\pi_I(f(p)p-f(q)q)\in C_I^f$.  Thus $C_L^f\to C_I^f$ is a surjection.  The
result  follows from the exact commutative diagram:
\[\xymatrix{ 0\ar[r]&0\\
(KQ/L)/C_L^f\ar[u]\ar[r]& (KQ/I)/C_I^f\ar[u]\\
KQ\ar[u]\ar[r]^=&KQ\ar[u]\\%
C_L^f\ar[u]\ar[r]& C_I^f\ar[u]\\
0\ar[r]\ar[u]&0\ar[u]
}\]
\end{proof}

Letting $L=0$ in the above lemma, we see that if $I$ is an ideal contained in $J^2$, 
the quasi-commuting algebra of $KQ$ for some function $f$  maps onto the commuting algebra  of $KQ/I$ for the same $f$.

\section{Path connected components}

For the remainder of the paper, we will restrict our attention to the `hereditary' case;
that is, we assume $I=0$, and we study the quasi-commuting algebras
 of a path algebra $KQ$.

Next, consider the relation $\sim$ on the vertex set of  $Q$ given by $v\sim w$ if and only if there
are paths from $v$ to $w$ and from $w$ to $v$.  
 Then $\sim$ is easily seen to be an equivalence relation on  the vertex set  of $Q$.
The equivalence class of a vertex $v$, denoted $\wr v\wr$, is  called the \emph{path connected component of $v$}  .
It is well-known that different path connected components are disjoint and that the disjoint union of the path connected
components is the set of vertices.  Note that $v\in \wr v \wr$ since $v$ is a path of length $0$ with $v$ 
both the start and end vertex.  

 We will see in Theorem~\ref{thm:struct_thm} that the path connected components are related to the ring structure of the commuting algebra of $KQ$.  Before that the next few results show that in studying the structure of 
quasi-commuting algebras, we may assume that $Q$ has no loops or multiple arrows.

Our next goal is to describe the ring structure of a quasi-commuting algebra of $KQ$ and $f$.
In the hereditary case ($I=0$), we have $\pi\colon KQ\to KQ/I$ is the identity map.
 Thus by Corollary \ref{prop:path_implies_arrows}, we have that $vKQ/C^fw\ne 0$ if and only if there is a path $p$ in $Q$ from $v$ to $w$.

We fix the following notation: $Q $ is a quiver with $n$ vertices $v_1,\dots, v_n$ and 
$f:\cB\to K^*$ is a set map.  Let
$\cD=\{D_1,\dots, D_m\}$  be the path connected components of $Q$. For
$i=1,\dots, m$, set  $d_i = |D_i|$ and hence, $\sum_{i=1}^md_i=n$. 
\\ \begin{definition}\rm{
If $a$ and $b$ are positive integers,   $M_{a\times b}(K)$ denotes the $a\times b$  matrix with
each entry $K$.  We also define $M_{a\times b}(0)$ to be the $a\times b$ matrix, all of whose entries
are $0$. 
}
\end{definition}

The next result will be applied in the Structure Theorem below.
\begin{proposition}\label{prop:struct}
 Let $v$ and $w$ be vertices in $Q$ with $v$ in $D_i$ and $w$ in $D_j$. 
\begin{enumerate}
\item If $i\ne j$, and  there is a path in $Q$ from $v$ to $w$, then  there is no path in $Q$ from any vertex 
in $D_j$  to any vertex in $D_i$.
\item There is a path in $Q$ from $v$ to $w$ if and only if  there are paths in $Q$ from each vertex
in $D_i$ to each vertex in $D_j$.
\item the $d_i\times d_i$-matrix with entries $v_i(KQ/C^f)v_i$  is $M_{i,i}(K)$, the 
 $d_i\times d_i$-matrix ring with entries $K$.
\item if $i\ne j$ and there are no paths from $D_i$ to $D_j$, then
 the  $d_i\times d_j$-matrix with entries $v_i(KQ/C^f)v_j$ is $M_{i,j}(0) $.   
\item if $i\ne j$, and there is a path from a vertex in $D_i$ to a vertex in $D_j$
the $d_i\times d_j$-matrix with entries $v_i(KQ/C^f)v_j$ is $M_{i,j}(K) $
and the  $d_j\times d_i$-matrix with entries $v_j(KQ/C^f)v_i$ is $M_{j,i}(0) $.                                                          
\end{enumerate} 
\end{proposition}   \begin{proof} \begin{enumerate}
\item Since $i\ne j$, $D_i\cap D_j=\emptyset$.  Suppose there are  paths from a vertex
$v\in D_i $ to $w\in D_j$ and  from a vertex
$w'\in D_j$ to $v'\in D_i$.  But  since there are paths from $w$ to $w'$ and from $v'$ to $v$, concatenating the paths,
we obtain a cycle at $v$ containing  $w$.  Thus $v$ and $w$ are in the same connected component - a contradiction.

\item If $v, v'\in D_i$ and $w,w'\in D_ j$,  and $v\extto{p}w$ is a path in $Q$, then there is a path $v' \to v\extto{q}w\to w'$.

\item Use (2), that $v(KQ/C^f)v$ is 1-dimensional, and that $D_i$ is path connected.

\item Clear.

\item Follows from (1), (2) and (4).

\end{enumerate}
\end{proof}

\begin{definition}\label{def:block}{\rm Given positive integers $d_1,\dots, d_m$ and  $n=\sum_{i=1}^md_i$.
An $n\times n$ display of $0$s and $K$s is said to be in \emph{$(d_1,\dots,d_m)$ block form},
if 
  \[A = \begin{pmatrix}
B_{1,1}&B_{2,2}&\cdots&B_{2,m}\\
\vdots&\vdots&&\vdots\\
B_{m,1}&B_{m,2}&\cdots&B_{m,m}\end{pmatrix}\]\\
\\where each $B_{i,i}$ is the  $d_i\times d_i$ matrix ring with entries $K$, and
 each $B_{i, j}$ is a $d_i\times d_j $-matrix with entries either 0 or $K$.  }
\end{definition}

For the rest of the paper,
 we assume that $f(p)=1$ for all paths $p$ in $\cB$.  For this choice of $f$, we let $C= C^f$ and
recall that we say that $KQ/C$ is the commuting algebra of $KQ$.

The next result shows that the commuting algebra of $KQ$ is isomorphic to a subring of
$n\times n $ matrix ring with entries $K$.

\begin{theorem}[Structure Theorem]\label{thm:struct_thm} Let $Q$ be a finite quiver with $n$
vertices, $\{v_1,\dots,v_n\}$.  Let $D_1,\dots, D_m$ be the path connected components of $Q$ with 
$|D_i|=d_i$,
for $i=1,\dots,m$. Reorder the vertices so that the first $d_1$ vertices are the vertices in $D_1$, 
the next $d_2$ vertices are the vertices in $D_2, \dots$, the last $d_m$ vertices are the vertices in $D_m$.
The commuting algebra of $KQ$ is in $(d_1,\dots, d_m)$ block form
with $B_{i,j} =M_{i,j}(K)$ if there is a path in  $Q$ from some vertex in $  D_i$ to a vertex in $D_j$
and 	 $B_{i,j} =M_{i,j}(0)$ if there is no  path in  $Q$ from any vertex in $  D_i$ to a vertex in $D_j$.

	\end{theorem}

\begin{proof} The proof directly follows from Proposition~\ref{prop:struct}.
\end{proof}

\begin{definition}\label{def:vert_ord}{\rm Keeping the notation of the theorem, if the vertices are ordered in such a fashion that the first $d_1$ vertices are the vertices in $D_1$, 
the next $d_2$ vertices are the vertices in $D_2$,  ..., the last $d_m$ vertices are the vertices in $D_m$,
we say the ordering of the vertices is \emph{ consistent} with $\cD$.}\end{definition}

\section{Morita equivalence}

Using  Morita  equivalence,   we find a basic algebra in the Morita class of a commuting algebra.   Recall that
a finite dimensional $K$-algebra, $\Lambda$, is called \emph{basic} if every simple  right $\Lambda$-module is 
1-dimensional.    Note that  a finite dimensional basic  $K$-algebra is unique, up to ring isomorphism, in its Morita class.  
We choose a an  ordering of  the vertices of $Q$ consistent with  $\cD$.

Select one vertex, $w_i\in D_i$, from each $D_i$.  Let 
$P=\oplus_{i=1}^m \rho(w_i)KQ/C$.
Note that each $\rho(w_i)$ is a nonzero idempotent in $KQ/C$ since $\rho$ is a ring homorphism and
$w_i\not\in C$ by Lemma \ref{lem:lin_alg}.

\begin{lemma}\label{lem:projgen}  The right $\kq$-module $P$ is a right projective generator for 
$\smod(\kq)$.\end{lemma}   

\begin{proof}Clearly $P$ is a projective $KQ/C$-module. Let $v$ be a vertex in $Q$. To show that
$P$  is a generator we need to show that every indecomposable projective module, $\rho(v)\kq$, is isomorphic to one of the $\rho(w_i)\kq$. 
We have that $v\in D_i$, for some $i$. 
 Thus there is a path $p$ from $v$ to $w_i$ and $q$ from $w_i$ to $v$ since both
$v$ and $w_i$ are in the same path connected component. 
Consider $\rho(pq)+C\in \rho(w_i)KQ/C$ and $\rho(qp)+C\in \rho(v)KQ/C$. 
Then $qp-w_i \in C$ and $pq-v\in C$.  Hence $\rho(pq)=\rho(v)$ and $\rho(qp)=\rho(w_i)$.
It follows that  multiplication on the right by the elements $\rho(p)+C$ and $\rho(q)+C$ induce the desired inverse isomorphisms. 
\end{proof}

\begin{theorem}\label{thm:morita}Let $Q$ be  a quiver with $n$ vertices $\{v_1,\dots v_n\}$ and let	$D_1,\dots D_m$
be the path connected components of $Q$. Set $P=\oplus_{j=1}^m\rho(w_j)KQ/C$, where, for  each $j=1,\dots, m$,  $w_j$ is a vertex in $D_j$.
 Then the  $K$-algebra $\End_{KQ/C}(P)$ is a basic algebra in the 
Morita  class of $KQ/C$.

\end{theorem}

\begin{proof}
By Lemma \ref{lem:projgen}, we have that      $P$ is a finitely generated projective generator for the category of finitely generated right $KQ/C$-modules.  By Morita equivalence, 
the category of finitely generated right $KQ/C$-modules is equivalent to
the category of finitely generated  right $\End_{KQ/C}(P)$-modules. 
Since $\{\rho(w_1),\dots, \rho(w_m)\}$ is a full set of
orthogonal non-isomorphic idempotents in $\kq$,  $\End_{\kq}(P)$ is a basic algebra Morita equivalent to $\kq$.
\end{proof}

  We immediately have the following. 

\begin{proposition}\label{prop:end_P} The following results hold:
\begin{enumerate}
\item For $1\le i\le m$, 	$\End_{\kq}\rho(w_i)(\kq	)\rho(w_i)$ is isomorphic to $K$.
\item 		For $1\le i,j\le m$	with $i\ne j$, 	$\End_{\kq}\rho(w_i)(\kq	)\rho(w_j)$ is isomorphic to 
$K$
if and only if there is a path $p$ in $Q$ from $w_j$ to $w_i$. Otherwise, 	
$\End_{\kq}\rho(w_i)(\kq	)\rho(w_j)=0$.
\item For $1\le i,j\le m$	with $i\ne j$, if $\End_{\kq}\rho(w_i)(\kq	)\rho(w_j)$ is isomorphic to $K$,
then $\End_{\kq}\rho(w_j)(\kq	)\rho(w_i)=0$.
\end{enumerate}
\end{proposition}
\begin{proof} Since, for each $1\le i\le m$, $w_i\in D_i$, the result follows from applying Proposition
\ref{prop:struct} 
 
\end{proof}

\begin{definition}{\rm  We call $\End_{\kq}(P)$ the \emph{skeleton of $KQ$} and denote it by $Sk(Q)$. }
\end{definition}

 In the next section we investigate the structure of $Sk(Q)$.

\section{The skeleton of an algebra} 

For every finite dimensional algebra, $\Lambda = KQ/I $, the \emph{skeleton} of $\Lambda$ is a basic algebra in the Morita equivalence class of the commuting algebra of $\Lambda$. 
We believe that, in general, the skelton of an algebra contains basic structural information about the algebra. In this section, we only deal with path algebras; that is,  $I=0$, unless otherwise stated.

  In the construction of the commuting algebra of $KQ$,   we see that there are 
$m$ paths connected  components of $Q$,     each corresponding  to a vertex of the  skeleton,
$Sk(Q)$, of  $KQ$. Let $\{x_1, \dots, x_m\}$ be the vertex set  of $Sk(Q)$ where each $x_i$ corresponds to a path connected component $D_i$ of $Q$.

\begin{proposition}
Let $KQ/C$ be the commuting algebra of $KQ$.  Then the commuting algebra of $KQ/C$ is
isomorphic to $KQ/C$. Moreover, the commuting algebra of $Sk(Q)$ is isomorphic to $Sk(Q)$. 
\end{proposition}

\begin{proof}
We show that the commuting ideal of $KQ/C$ is $(0)$.  Let $p\|q$ be parallel 
 paths in $Q$
and $\rho\colon KQ \to KQ/C$ be the canonical surjection.  Then the commuting ideal of
$KQ/C$ is generated by $\{ \rho(p)-\rho(q)$ with $p\|q\}$.
But $\rho(p)-\rho(q)=\rho(p-q)=0$, since $p-q\in C$.
\end{proof}

\begin{proposition}\label{prop:partial}  Setting $x_i \le x_j$   if and only if there is a path from
$w_i \in D_i$ to $w_j \in D_j$ in $Q$, is a partial ordering of the vertex set of $Sk(Q)$.
\end{proposition}

\begin{proof}  Let $1\le i,j\le m$ with $i\ne j$.
Suppose that $x_i\le x_j$.  We need to show that
${x_j \not\le x_i}$. But it follows directly from Proposition~\ref{prop:struct} that  there is  no non-zero path from $x_j$ to $x_i$ and hence ${x_j \not\le x_i}$. Note that if $i = j$ then $x_i = x_j$. 
\end{proof}

\begin{definition} If $S$ is a finite partially ordered set, then $s_{i_1}<s_{i_2}<\cdots <s_{i_t}$, for $s_{i_j} \in S$, 
 is called a \emph{chain of length $t$}.
\end{definition}

\begin{theorem}\label{thm:fin_dim} Let $Q$ be a finite quiver. Then the commuting algebra of
$KQ$ and  $Sk(Q)$ are   finite dimensional $K$-algebras  with finite global dimension $\le$ the length
of the longest chain of vertices of $Sk(Q)$.

\end{theorem} 

\begin{proof}  Let $\{x_1\dots,x_m\}$ be the vertices of $Sk(Q)$.  By a proof similar to the proof of 
Theorem \ref{thm:nh-dim_1}, the skeleton of $Q$ is finite dimensional. 
By Proposition \ref{prop:partial} $\{ x_1,\dots, x_n\}$ are partially ordered by $x_i\le x_j$ if either $i=j$ or
there is a path in  $Q$ from a vertex in $D_i$ to a vertex in $D_j$.
By standard arguments it follows that the global dimension of the skeleton of $KQ$ is less than
or equal to the longest chain of vertices.  The result
follows since the commuting algebra of $Q$ is Morita equivalent to the skeleton of $Q$ and
global dimension is preserved by Morita equivalence.
\end{proof}

Let $(S,\preceq)$ be a finite partially ordered set.
 We say  that, for $s$ and $t$ in $S$, that $s$ is an \emph{immediate pedecesor } of $t$ or 
 $t$ is an \emph{immediate successor of $s$ if $s \prec t$}, and there is no $u\in S$ such that
$s\prec u\prec t$.  The \emph{incidence  algebra of the partially ordered set
$(S,\preceq)$} is the commuting algebra, denoted $Incid(S)$, of the quiver $Q(S)$, where $S$ is the vertex set of $(Q(S),\preceq)$
 and there is an arrow from vertex $s$ to vertex $t$ if $t$ is an immediate successor of $s$.

Note that the incidence algebra of a partially ordered set  is a uniquely defined  algebra.

\begin{theorem}\label{thm:skel_incid}  Let $Sk(Q)$ be  the skeleton of $KQ$, for a finite quiver $Q$.	Let $V$ be                         the vertex set   of $Sk(Q)$.  Viewing $V$ as a partially ordered set, the incidence algebra $Incid(V)$ is isomorphic
to $Sk(Q)$.
\end{theorem}

\begin{proof} 
Since $Sk(Q)$ is finite dimensional, with
one dimensional simple modules, it is isomorphic
to the quotient of a path algebra $KQ^*/I^*$.  Let $V$ be the vertex set of $Q^*$.
By Proposition
 \ref{prop:partial}, $(V,\preceq)$ is a partially ordered set, and we let $Incid(V)$ be its incidence    
algebra.  The vertices of both  $Incid(V)$ and $Sk(Q)$
are the same.   The algebra $Sk(Q)$  has one vertex  for each path connected
component of $Q$.  There is  an arrow from  $v$ to $w$ if there is a path from 
$v$ to $w$ in $Sk(Q)$.        On the other hand, there is  an arrow from  $v$ to $w$ in $Incid(V)$ if $v\prec w$ that is if $v$ is a immediate predecessor of $w$. It is an easy exercise check that there
is a $K$-algebra isomorphism from $Sk(Q)$ to $Incid(V)$. 
\end{proof}

As is well-known,  incidence algebras and algebras of partially ordered finite sets
 are isomorphic, see, for example \cite[Section 1.2]{IM} or more generally, also see \cite{S, SO}.  Recall that the incidence algebra of a partially ordered set $P$ is
 the $K$-algebra with basis given by elements $p^x_y$ whenever two elements $x,y \in P$ are such that $x < y$. The multiplication of basis elements is given by $p^x_y p^w_z= p^x_z$ if $y =w$ and the product is zero otherwise. 
  On the other hand the algebra of a partially ordered set is the quotient $KQ/I$ of the path algebra of a quiver $Q$ where the vertices of the quiver are the elements of the partially ordered set $P$ and  there is an arrow from vertex $v$ to vertex $w$ if $v < w$  and if there is no other element $s$ in $P$ such that $v < s  <w$ unless $s=v$ or $s =w$. The ideal $I$ is  generated by $p -q$ for any parallel paths $p$ and $q$ in $Q$ of length at least two. These can be viewed as the commuting algebras of path algebras whose quiver is a
 Hasse diagram. Since the skeleton of an algebra is Morita equivalent to the commuting algebra of $KQ$, for any path algebra,  the category of finitely generated 
modules over the skeleton of the commuting algebra of $KQ$ embeds into the 
the category of  $KQ$-modules. This observation allows one to apply the
rich theory of representations of partially ordered sets to study of a `piece' of
the (usually wild) category of $KQ$-modules.

\section{Examples}

 \begin{Example}\label{ex:4x2}
\rm{
Let $Q$ be the quiver:

$$\xymatrix{
v_1\ar[rr]\ar[rd]&&v_2\ar[ddd]\\
&v_5\ar@/^/[d]\\
&v_6\ar@/^/[u]\\
v_4\ar[uuu]&&v_3\ar[ll]\ar[lu]
}$$
} There are two path connected components $$D_1=\{v_1,v_2,v_3,v_4\} \mbox{ and }
D_2=\{v_5,x_6\}.$$  The commuting algebra for $KQ$ in block form is:
$$
\begin{pmatrix}
K&K&K&K&K&K\\
K&K&K&K&K&K\\
K&K&K&K&K&K\\
K&K&K&K&K&K\\
0&0&0&0&K&K\\
0&0&0&0&K&K
\end{pmatrix}
$$
  and the skeleton of $KQ$ has 2 vertices corresponding to $D_1$ and $D_2$ and an arrow
from ``$D_1$" to ``$D_2$"

The skeleton is the incidence algebra of $w_1\to w_2$.  The incidence algebra is isomorphic to
$$\begin{pmatrix}K&K\\0&K\end{pmatrix}$$
\end{Example}

\begin{Example}{\rm Consider the quivers $$ Q_1 =\xymatrix{  1\ar[r]^a&2\ar[r]^{b}& 3\ar[d]^c\\
6\ar[u]^f&5\ar[l]^e&4\ar[l]^d
}  \mbox{ and     }
Q_2=\xymatrix{  1\ar[r]^a&2\ar[d]^g\ar[r]^{b}& 3\ar[d]^c\\
6\ar[u]^f&5\ar[l]^e&4\ar[l]^d }$$

Both $Q_1$ and $Q_2$ have the same commuting algebra, namely, the  $6\times 6$ matrix ring and their  skeleton is given by a single vertex.  But $KQ_1$ and $KQ_2$ are not isomorphic.  The skeleton of both $KQ_1$ and $KQ_2$
 is a single vertex.
}\end{Example}

\begin{Example}\label{ex:6x6}{\rm
                 Let $Q=\xymatrix{
1\ar[r]^a&2\ar[d]^g\ar[r]^{b}& 5\ar[d]^c\\
4\ar[u]^f&3\ar[l]^e\ar[r]^d&6}
$  The path connected components are $D_1=\{1,2,3,4\},  D_2=\{ 5\}$, and $ D_3= \{6\} $ and  the commuting algebra $KQ/C$ is the $6\times 6$ matrix ring 
\[KQ/C=\begin{pmatrix}  K &K& K& K &K&K   \\K                 
&K   &K&K &K&K\\K                      
&K   &K&K &K&K\\K                      
&K   &K&K &K&K\\0&0&0&0&K&K\\0&0&0&0&0&K
\end{pmatrix}\]  The skeleton of
$KQ$ is $w_1\to w_2\to w_3$.
}
\end{Example}

The next example shows that if $f(p)=1$ but $I\ne 0$, then  the global dimension of  the commuting 
algebra can be infinite.

\begin{Example}\label{ex:inf-gl}   {\rm First take $I=0$.
and  $Q=\xymatrix{&v_1\ar[dr]^a\\
v_3\ar[ur]^c&&v_2\ar[ll]^b}$
The commuting algebra of $Q$ is the $3\times3$ matrix ring with entries in $K$.   We know that the skeleton of the $3\times 3$ matrix ring with entries in $K$ is  just $K$. 

Continuing, 
we now consider $KQ/I$ where $I=\langle ab,bc,ca \rangle$. 
In this case, $KQ/I$ is a monomial algebra and has no paths of length 2. It is well-known that this algebra has infinite global dimension.
Then the global dimension of $((KQ/I)/C_{KQ/I})$ is also infinite since $C_{KQ/I}=0$.

}\end{Example}

We give some further examples of commuting algebras and skeletons.
\begin{Example}
Let $Q$ be the quiver that has two vertices $v$ and $w$ and $n$ arrows from $v$ to $w$. Then the commuting algebra of $Q$ isomorphic to the skeleton of $KQ$ and consists of two vertices  and one arrow from $v$ to $w$. 
\end{Example}

\begin{Example} 
If the underlying graph of a quiver $Q$ is a tree then the commuting algebra of $Q$ is isomorphic to  the skeleton of $KQ$ and is the algebra is itself. This follows from Proposition \ref{prop:partial}.
\end{Example}

\begin{Example}
Let $Q$ be an oriented cycle with $n$ vertices and $n$ arrows. Then the commuting algebra
of $Q$ is isomorphic to the $n \times n$-matrix ring whereas  the skeleton of $Q$ corresponds  to a vertex with no arrows. So in this case the commuting algebra of $KQ$ is not isomorphic to the skeleton of $KQ$. 
\end{Example}

\bibliographystyle{alph}


%

\end{document}